\documentclass[a4paper]{amsart}
\usepackage{graphicx}
\usepackage{amssymb}
\usepackage{amsmath}
\usepackage{amsthm}
\usepackage{amscd}
\usepackage[all,2cell]{xy}

\UseAllTwocells \SilentMatrices
\newtheorem{thm}{Theorem}[section]

\newtheorem{cor}[thm]{Corollary}
\newtheorem{lem}[thm]{Lemma}

\newtheorem{prop}[thm]{Proposition}
\theoremstyle{definition}

\theoremstyle{remark}

\numberwithin{equation}{section}

\begin{document}
\title[Representability and autoequivalence groups]{Representability and autoequivalence groups}
\author[Xiao-Wu Chen] {Xiao-Wu Chen}

\date{\today}

\thanks{E-mail: xwchen$\symbol{64}$mail.ustc.edu.cn}

\maketitle

\dedicatory{}%
\commby{}%

\begin{abstract}
For a finite dimensional algebra $A$, we prove that the bounded homotopy category of projective $A$-modules and the bounded derived category of $A$-modules are dual to each other via certain categories of locally-finite cohomological functors. The duality gives rise to a $2$-categorical duality between certain strict $2$-categories involving the bounded homotopy categories and bounded derived categories, respectively. We apply the $2$-categorical duality to the study of triangle autoequivalence groups. These results are analogous to the ones in [M.R. Ballard, {\em Derived categories of sheaves on singular schemes with an application to reconstruction}, Adv. Math. {\bf 227} (2011), 895--919].
\end{abstract}

\section{Introduction}

Let $k$ be field. It is well known that the homological behavior of a finite dimensional $k$-algebra $A$ with infinite global dimension is similar to that of a singular projective scheme $\mathbb{X}$. For example, the difference between the category ${\rm perf}(\mathbb{X})$ of perfect complexes and the bounded derived category $\mathbf{D}^b({\rm coh}\mbox{-}\mathbb{X})$ measures the singularity of $\mathbb{X}$; see \cite{Or04}. In the same manner, the difference between the bounded homotopy category $\mathbf{K}^b(A\mbox{-proj})$ of projective $A$-modules and the bounded derived category $\mathbf{D}^b(A\mbox{-mod})$ measures the homological singularity of $A$, or more precisely, the stable properties of the module category $A\mbox{-mod}$; see \cite{Buc, Kr}.

The following remarkable observation is made in \cite{Ball}: for such a scheme $\mathbb{X}$, there is a duality of linear categories between  ${\rm perf}(\mathbb{X})$ and $\mathbf{D}^b({\rm coh}\mbox{-}\mathbb{X})$ via the categories of cohomological functors. This duality is applied to the study of triangle autoequivalence groups and the reconstruction of $\mathbb{X}$ from these triangulated categories. We mention that  the duality is essentially proved by the representability of certain locally-finite cohomological functors. The related representability theorems are obtained in \cite{BV, CKN, Rou}.

Inspired by \cite{Ball}, it is natural to expect that such a duality holds between $\mathbf{K}^b(A\mbox{-proj})$ and $\mathbf{D}^b(A\mbox{-mod})$. The first result confirms this expectation; see Theorem \ref{thm:1}. We mention that its proof is modified from the one in \cite[Section 3]{Ball}.

Following \cite[Section 4]{Ball}, we use the pseudo-adjunctions and the above duality to obtain a $2$-categorical duality, which involves these triangulated categories. For a more precise statement of the following second result, we refer to Theorem \ref{thm:2}.

\vskip 5pt

\noindent {\bf Theorem}. {\em Let $\mathbb{K}^b$ be the strict $2$-category with objects being all finite dimensional algebras $A$, $1$-morphisms being triangle functors between $\mathbf{K}^b(A\mbox{-}{\rm proj})$, and $2$-morphisms being natural transformations. Let $\mathbb{D}^b$ be the analogous $2$-category replacing  $\mathbf{K}^b(A\mbox{-}{\rm proj})$ by $\mathbf{D}^b(A\mbox{-}{\rm mod})$. Then there is a $2$-categorical duality
$$\mathbb{K}^b\stackrel{\sim}\longrightarrow \mathbb{D}^b,$$
which acts on objects by the identity. }

\vskip 5pt

We mention that an analogue of the above theorem for projective schemes is also true by the results in \cite[Section 4]{Ball}.

The above $2$-categorical duality is applied to the study of triangle autoequivalence groups. For a triangulated category $\mathcal{T}$, we denote by ${\rm Aut}_\vartriangle(\mathcal{T})$ its triangle autoequivalence groups, whose elements are the isomorphism classes of triangle autoequivalences on $\mathcal{T}$. The derived Picard group ${\rm DPic}(A)$ is an important invariant of an algebra $A$, whose elements are the isomorphism classes of two-sided tilting complexes over $A$; see \cite{Ye, RZ}.

The following group homomorphisms are well known
$${\rm DPic}(A)\stackrel{\rm ev}\longrightarrow {\rm Aut}_\vartriangle(\mathbf{D}^b(A\mbox{-mod})) \stackrel{\rm res} \longrightarrow {\rm Aut}_\vartriangle(\mathbf{K}^b(A\mbox{-proj})).$$
Here, the homomorphism ``${\rm ev}$" sends a two-sided tilting complex $X$ to the derived tensor functor $X\otimes^\mathbb{L}_A-$, and ``${\rm res}$" denotes the restriction of autoequivalences. Moreover, the homomorphism ``${\rm ev}$" is injective. By the proof of \cite[Section 6]{CY}, the homomorphism ``${\rm res}$" is also injective.

The fundamental open question in \cite[Section 3]{Ric} asks whether any derived equivalence is standard, or equivalently, whether ``${\rm ev}$" is surjective. The following third result implies that the open question is equivalent to the surjectivity of the composition ``${\rm res}\circ {\rm ev}$"; see Corollary \ref{cor:auto}.

\vskip 5pt

\noindent {\bf Proposition}. {\em Let $A$ be a finite dimensional $k$-algebra. Then the restriction homomorphism
$${\rm Aut}_\vartriangle(\mathbf{D}^b(A\mbox{-}{\rm mod})) \stackrel{\rm res} \longrightarrow {\rm Aut}_\vartriangle(\mathbf{K}^b(A\mbox{-}{\rm proj}))$$
is an isomorphism. }
\vskip 5pt

The surjectivity of ``${\rm res}$" is equivalent to the  fact that any triangle autoequivalence on $\mathbf{K}^b(A\mbox{-}{\rm proj})$ extends to a triangle autoequivalence on $\mathbf{D}^b(A\mbox{-}{\rm mod})$. More generally, the extension of triangle functors is studied in Proposition \ref{prop:extension}.  The relation between the isomorphism ``${\rm res}$" and the work \cite{CY} is discussed at the end of this paper; see Corollary \ref{cor:stand}.

The structure of this paper is straightforward. Throughout, we require that all the algebras, categories and functors are $k$-linear over a fixed field $k$.

\section{Cohomological functors and representability}

In this section, we prove that there is a duality between the bounded homotopy category of projective modules and the bounded derived category of modules. The duality is realized via the categories of  locally-finite cohomological functors.

\subsection{A representability lemma}

Let $k$ be a field. Denote by $k\mbox{-Mod}$ the category of $k$-vector spaces and by $k\mbox{-mod}$ the full subcategory of finite dimensional vector spaces.

Let $\mathcal{C}$ be a skeletally small  triangulated category, which is $k$-linear and Hom-finite. Here, the Hom-finiteness means that each Hom space ${\rm Hom}_\mathcal{C}(X, Y)$ is finite dimensional. A $k$-linear functor $F\colon \mathcal{C}\rightarrow k\mbox{-Mod}$ is \emph{cohomological} provided that $F$ sends exact triangles to long exact sequences of vector spaces.  The cohomological functor $F$ is \emph{locally-finite} provided that the vector space $\bigoplus_{n\in \mathbb{Z}}F(\Sigma^nX)$ is finite dimensional for each object $X\in \mathcal{C}$. Denote  by ${\rm coho}(\mathcal{C})$ the category of locally-finite cohomological functors.

Let $\mathcal{T}$ be a triangulated category with arbitrary coproducts. Recall that an object $X$ is compact provided that the following canonical injection
$$\bigoplus_{i\in \Lambda} {\rm Hom}_\mathcal{T}(X, Y_i)\longrightarrow {\rm Hom}_\mathcal{T}(X, \bigoplus_{i\in \Lambda} Y_i)$$
is surjective for any objects $Y_i$ indexed by any set $\Lambda$. Denote by $\mathcal{T}^c$ the full subcategory of $\mathcal{T}$ formed by compact objects; it is  a thick triangulated subcategory. The triangulated category $\mathcal{T}$ is said to be \emph{compactly generated} provided that $\mathcal{T}^c$ is skeletally small and that for each nonzero object $Y\in \mathcal{T}$, there is a nonzero morphism $X\rightarrow Y$ with $X$ compact.

We assume further that $\mathcal{T}$ is $k$-linear. An object $X$ is \emph{locally-finite} provided that the restricted Hom functor
$${\rm Hom}_\mathcal{T}(-, X)|_{\mathcal{T}^c}\colon (\mathcal{T}^c)^{\rm op}\longrightarrow k\mbox{-Mod}$$
is locally-finite. Denote by $\mathcal{T}_{\rm lf}$  the full subcategory of $\mathcal{T}$ consisting of locally-finite objects, which is a thick triangulated subcategory.

The following fundamental result \cite[Lemma 3.3]{Ball} is a finite version of the Brown representability theorem. Its first part is due to \cite[Lemma 2.14]{CKN}, and its second part relies on \cite[Section 2]{Kr00}.

\begin{lem}\label{lem:rep}
Let $\mathcal{T}$ be a $k$-linear triangulated category which is compactly generated. Then any cohomological functor $F\colon (\mathcal{T}^c)^{\rm op}\longrightarrow k\mbox{-{\rm mod}}$ is represented by some object $X\in \mathcal{T}$, that is, isomorphic to ${\rm Hom}_\mathcal{T}(-, X)|_{\mathcal{T}^c}$. Moreover, any natural transformation between such cohomological functors is induced by some morphism between the representing objects. \hfill $\square$
\end{lem}

Recall that a morphism $f\colon X\rightarrow Y$ in $\mathcal{T}_{\rm lf}$ is \emph{phantom} provided that each composition $f\circ g$ is zero for any morphism $g\colon C\rightarrow X$ with $C$ compact. These phantom morphisms form a two-sided ideal $\mathbf{ph}$ of $\mathcal{T}_{\rm lf}$. Denote by $\mathcal{T}_{\rm lf}/\mathbf{ph}$ the factor category of $\mathcal{T}_{lf}$ by the phantom ideal.

\begin{cor}\label{cor:coho}
Let $\mathcal{T}$ be a $k$-linear triangulated category which is compactly generated. Then the restricted Yoneda functor
$$\mathcal{T}_{\rm lf} \longrightarrow {\rm coho}((\mathcal{T}^c)^{\rm op}), \quad X\mapsto {\rm Hom}_\mathcal{T}(-, X)|_{\mathcal{T}^c}$$
is full and dense. In particular, it induces an equivalence of categories
$$\mathcal{T}_{\rm lf}/\mathbf{ph} \stackrel{\sim}\longrightarrow {\rm coho}((\mathcal{T}^c)^{\rm op}).$$
\end{cor}

\begin{proof}
The first assertion follows from Lemma \ref{lem:rep}. It suffices to observe by definition that a morphism $f$ is phantom if and only if ${\rm Hom}_\mathcal{T}(-, f)|_{\mathcal{T}^c}=0$.
\end{proof}

\subsection{Duality via cohomological functors}

Let $A$ be a finite dimensional $k$-algebra. Denote by $A\mbox{-Mod}$ the category of left $A$-modules. We denote by $A\mbox{-mod}$ and $A\mbox{-proj}$ the full subcategories consisting of finitely generated $A$-modules and finitely generated projective $A$-modules, respectively.

We use the cohomological notation. Denote a complex of $A$-modules by $X=(X^n, d_X^n)_{n\in \mathbb{Z}}$. The $n$-th cohomology of $X$ is denoted by $H^n(X)$. For each $n$, we denote by $\sigma_{\geq n}(X)$ the subcomplex of $X$ consisting of components with degree at least $n$. Denote by $\mathbf{K}(A\mbox{-Mod})$ and $\mathbf{D}(A\mbox{-Mod})$ the homotopy category and derived category of $A\mbox{-Mod}$, respectively. The translation of complexes is denoted by $\Sigma$.

We collect some well-known facts for later use. The following observation is contained in \cite[Lemma 2.6]{KrYe}.

\begin{lem}\label{lem:ht1}
Let $f\colon P\rightarrow X$ be a chain morphism such that $P$ consists of projective modules and $H^n(X)=0$ for $n<0$. Assume that the restriction $f|_{\sigma_{\geq 0}(P)}\colon \sigma_{\geq 0}(P)\rightarrow X$ is homotopic to zero. Then $f$ is homotopic to zero.
\end{lem}

\begin{proof}
We apply the cohomological functor ${\rm Hom}_{\mathbf{K}(A\mbox{-}{\rm Mod})}(-, X)$ to the canonical triangle
$$\sigma_{\geq 0}(P) \longrightarrow P \longrightarrow P/{\sigma_{\geq 0}(P)} \longrightarrow \Sigma\sigma_{\geq 0}(P).$$
By assumption, we observe that ${\rm Hom}_{\mathbf{K}(A\mbox{-}{\rm Mod})} (P/{\sigma_{\geq 0}(P)}, X)=0$. We deduce  that the restriction map
$${\rm Hom}_{\mathbf{K}(A\mbox{-}{\rm Mod})}(P, X)\longrightarrow {\rm Hom}_{\mathbf{K}(A\mbox{-}{\rm Mod})}(\sigma_{\geq 0}(P), X)$$ is injective. Then the result follows.
\end{proof}

For an $A$-module $M$, we denote by $\mathbf{i}(M)$ the injective resolution of $M$. Then we have a quasi-isomorphism $a_M\colon M\rightarrow \mathbf{i}(M)$, where $M$ is viewed as a stalk complex concentrated on degree zero.

\begin{lem}
Let $X$ be a complex consisting of injective $A$-modules. Then there is an isomorphism
\begin{align}\label{iso:inj1}
{\rm Hom}_{\mathbf{K}(A\mbox{-}{\rm Mod})}(\mathbf{i}(M), X)\longrightarrow {\rm Hom}_{\mathbf{K}(A\mbox{-}{\rm Mod})}(M, X), \quad f\mapsto f\circ a_M.
\end{align}
In particular, we have an isomorphism
\begin{align}\label{iso:inj2}
{\rm Hom}_{\mathbf{K}(A\mbox{-}{\rm Mod})}(\mathbf{i}(A), \Sigma^n(X))\stackrel{\sim}\longrightarrow H^n(X)
\end{align}
 for each integer $n$.
\end{lem}

\begin{proof}
The  first isomorphism is due to \cite[Lemma 2.1]{Kr}. For the second, we just use the canonical isomorphism
${\rm Hom}_{\mathbf{K}(A\mbox{-}{\rm Mod})}(A, \Sigma^n(X))\simeq H^n(X).$
\end{proof}

We denote by ${\rm rad}(A)$ the Jacobson radical of $A$ and set $A_0=A/{{\rm rad}(A)}$.  For a complex $X$, we denote by $\tau_{<n}(X)$ and $\tau_{>n}(X)$ the good truncations. More precisely, we have $\tau_{<n}(X)= \cdots \rightarrow X^{n-3} \rightarrow X^{n-2} \rightarrow {\rm Ker}d_X^{n-1} \rightarrow 0$ and $\tau_{>n}(X)=0\rightarrow {\rm Cok} d_X^{n} \rightarrow X^{n+2} \rightarrow X^{n+3} \rightarrow \cdots$.

\begin{lem}\label{lem:split}
Let $X$ be a complex consisting of injective $A$-modules. Then the following statements are equivalent:
\begin{enumerate}
\item ${\rm Hom}_{\mathbf{K}(A\mbox{-}{\rm Mod})}(\mathbf{i}(A_0), X)=0$.
\item $H^0(X)=0$ and ${\rm Ker}d_X^{-1}$ is an injective $A$-module.
\item The complex $X$ is homotopic to $\tau_{<0}(X)\oplus \tau_{>0}(X)$.
\end{enumerate}
In this situation, the complex $\tau_{<0}(X)\oplus \tau_{>0}(X)$ also consists of injective $A$-modules.
\end{lem}

\begin{proof}
For ``$(1)\Rightarrow (2)$", we observe that $\mathbf{i}(A)$ is an iterated extension of direct summands of $\mathbf{i}(A_0)$ in $\mathbf{K}(A\mbox{-}{\rm Mod})$. It follows that ${\rm Hom}_{\mathbf{K}(A\mbox{-}{\rm Mod})}(\mathbf{i}(A), X)=0$. By (\ref{iso:inj2}) we have $H^0(X)=0$. We observe an isomorphism  $${\rm Ext}_A^1(A_0, {\rm Ker}d_X^{-1})\simeq {\rm Hom}_{\mathbf{K}(A\mbox{-}{\rm Mod})}(A_0, X),$$
 since  $0\rightarrow {\rm Ker}d_X^{-1} \rightarrow X^{-1} \rightarrow X^0 \rightarrow X^{1}$ is a part of an injective resolution of ${\rm Ker}d_X^{-1}$. Applying (\ref{iso:inj1}) for $M=A_0$ and using this isomorphism, we deduce ${\rm Ext}_A^1(A_0, {\rm Ker}d_X^{-1})=0$, which implies that ${\rm Ker}d_X^{-1}$ is an injective $A$-module.

For ``$(2)\Rightarrow (3)$", we observe that the $A$-modules ${\rm Im} d_X^{-1}$, ${\rm Im} d_X^0$ and ${\rm Cok} d_X^0$ are all injective. It follows that as a complex, $X$ is isomorphic to $\tau_{<0}(X)\oplus \tau_{>0}(X)\oplus E$, where $E=\cdots \rightarrow 0\rightarrow {\rm Im} d_X^{-1}\rightarrow X^0\rightarrow {\rm Im} d_X^0\rightarrow 0\rightarrow \cdots$ is homotopic to zero.  In view of (\ref{iso:inj1}) for $M=A_0$, the remaining implication ``$(3)\Rightarrow (1)$" is trivial. \end{proof}

\begin{cor}\label{cor:trunca}
Let $X$ be a complex consisting of injective $A$-modules and $n_0>0$. Assume that  ${\rm Hom}_{\mathbf{K}(A\mbox{-}{\rm Mod})}(\mathbf{i}(A_0), \Sigma^n(X))=0$ whenever $|n|\geq n_0$. Then $X$ is homotopic to $\tau_{<n_0}\tau_{>-n_0}(X)$, which is also consisting of injective $A$-modules.
\end{cor}

\begin{proof}
We apply Lemma \ref{lem:split} first to $\Sigma^{n}(X)$ for each $n\leq -n_0$. Then we have an isomorphism $X\simeq \tau_{<-n_0}(X)\oplus \tau_{>-n_0}(X)$ in $\mathbf{K}(A\mbox{-Mod})$, where $\tau_{<-n_0}(X)$ is acyclic with injective cocycles. It follows that $\tau_{<-n_0}(X)$ is homotopic to zero. Hence $X$ is homotopic to $\tau_{>-n_0}(X)$. Then we apply Lemma \ref{lem:split} to $\Sigma^{n}(\tau_{>-n_0}(X))\simeq \Sigma^n(X)$ for each $n\geq n_0$. By a similar reasoning, we obtain the required isomorphism in $\mathbf{K}(A\mbox{-Mod})$.
\end{proof}

The main concerns are the bounded homotopy category $\mathbf{K}^b(A\mbox{-proj})$ and the bounded derived category $\mathbf{D}^b(A\mbox{-mod})$. It is natural to view $\mathbf{K}^b(A\mbox{-proj})$ as a full triangulated subcategory of $\mathbf{D}^b(A\mbox{-mod})$; moreover, they are equal if and only if the algebra $A$ has finite global dimension.

The following intrinsic description  of the subcategory $\mathbf{K}^b(A\mbox{-proj})$ in $\mathbf{D}^b(A\mbox{-mod})$ is standard.

\begin{lem}\label{lem:perf}
Let $Y\in \mathbf{D}^b(A\mbox{-}{\rm mod})$. Then $Y$ lies in $\mathbf{K}^b(A\mbox{-}{\rm proj})$ if and only if ${\rm Hom}_{\mathbf{D}^b(A\mbox{-}{\rm mod})}(Y, \Sigma^i(X))=0$ for each  $X \in \mathbf{D}^b(A\mbox{-}{\rm mod})$ and $i\gg 0$. \hfill $\square$
\end{lem}

The main result of this section is as follows, which establishes the promised duality between $\mathbf{K}^b(A\mbox{-proj})$ and $\mathbf{D}^b(A\mbox{-mod})$. It is analogous to \cite[Theorem 3.2 and Proposition 3.12]{Ball}.

\begin{thm}\label{thm:1}
Let $A$ be a finite dimensional $k$-algebra. Then we have equivalences of categories
$$\mathbf{D}^b(A\mbox{-{\rm mod}})\stackrel{\sim}\longrightarrow {\rm coho}(\mathbf{K}^b(A\mbox{-}{\rm proj})^{\rm op}), \quad X\mapsto {\rm Hom}_{\mathbf{D}^b(A\mbox{-}{\rm mod})}(-, X)|_{\mathbf{K}^b(A\mbox{-}{\rm proj})}$$
and
$$\mathbf{K}^b(A\mbox{-}{\rm proj}) \stackrel{\sim}\longrightarrow {\rm coho}(\mathbf{D}^b(A\mbox{-{\rm mod}})), \quad P\mapsto {\rm Hom}_{\mathbf{D}^b(A\mbox{-}{\rm mod})}(P, -).$$
\end{thm}

\begin{proof}
For the first equivalence, we set $\mathcal{T}=\mathbf{D}(A\mbox{-Mod})$. It is well known that $\mathcal{T}$ is compactly generated and that there is a natural identification between $\mathcal{T}^c$ and $\mathbf{K}^b(A\mbox{-}{\rm proj})$.  Since $\mathbf{K}^b(A\mbox{-}{\rm proj})$ is generated by $A$, an object $X\in \mathcal{T}$ is locally-finite if and only if $\bigoplus_{n\in \mathbb{Z}}{\rm Hom}_\mathcal{T}(A, \Sigma^n(X))$ is finite dimensional. We recall  the canonical isomorphism
$${\rm Hom}_\mathcal{T}(A, \Sigma^n(X))\stackrel{\sim}\longrightarrow H^n(X).$$
It follows that a complex $X\in \mathcal{T}$ is locally-finite if and only if the total cohomogical space $\bigoplus_{n\in \mathbb{Z}}H^n(X)$ is finite dimensional, in other words, $X$ lies in $\mathbf{D}^b(A\mbox{-{\rm mod}})$. Hence, we identify $\mathcal{T}_{\rm lf}$ with $\mathbf{D}^b(A\mbox{-{\rm mod}})$.

We observe that there is no non-zero phantom morphism $f\colon X\rightarrow Y$ in $\mathbf{D}^b(A\mbox{-{\rm mod}})$. Indeed, we may assume that $X$ is a bounded-above complex of projective modules and that $f$ is a chain map. The phantom property implies that $f|_{\sigma_{\geq n}(X)}$ is homotopic to zero for any  integer $n$. By Lemma \ref{lem:ht1}, we infer that $f$ is homotopic to zero. By combining these facts, the first equivalence follows from Corollary \ref{cor:coho}.

For the second equivalence, let $A^{\rm op}$ be the opposite algebra of $A$. We consider $\mathcal{T}'=\mathbf{K}(A^{\rm op}\mbox{-Inj})$, the homotopy category of  injective $A^{\rm op}$-modules. By \cite[Proposition 2.3]{Kr}, $\mathcal{T}'$ is compactly generated and there is a natural identification between ${\mathcal{T}'}^c$ and $\mathbf{D}^b(A^{\rm op}\mbox{-mod})$. Recall that we identify a complex $Y$ in $\mathbf{D}^b(A^{\rm op}\mbox{-mod})$ with its injective resolution $\mathbf{i}(Y)$ in $\mathcal{T}'$.

We claim that an object $I$ in $\mathcal{T}'$ is locally-finite if and only if it lies in $\mathbf{K}^b(A^{\rm op}\mbox{-inj})$, the bounded homotopy category of finitely generated injective $A^{\rm op}$-modules. The ``if" part is clear. Conversely, we assume that $I$ is locally-finite. Then there is some $n_0>0$ such that  ${\rm Hom}_{\mathbf{K}(A^{\rm op}\mbox{-}{\rm Mod})}(\mathbf{i}(A_0), \Sigma^n(X))=0$ whenever $|n|\geq n_0$. By Corollary \ref{cor:trunca}, we may assume that $I$ is a bounded complex of injective $A$-modules. By (\ref{iso:inj2}) we infer that the total cohomology space $\bigoplus_{n\in \mathbb{Z}} H^n(I)$ is finite dimensional. It follows that the bounded complex $I$ is an injective resolution of a bounded complex of finitely generated $A$-modules. In other words, we have that up to isomorphism,  $I$ lies in $\mathbf{K}^b(A^{\rm op}\mbox{-inj})$. This proves the claim.

We now apply Corollary \ref{cor:coho} to $\mathcal{T}'$. We identify ${\mathcal{T}'}^c$ with $\mathbf{D}^b(A^{\rm op}\mbox{-mod})$,  and $\mathcal{T}'_{\rm lf}$ with $\mathbf{K}^b(A^{\rm op}\mbox{-inj})$. Since $\mathcal{T}'_{\rm lf}\subseteq {\mathcal{T}'}^c$, the phantom ideal vanishes. Hence, we have an equivalence
$$\mathbf{K}^b(A^{\rm op}\mbox{-inj})\stackrel{\sim}\longrightarrow {\rm coho}(\mathbf{D}^b(A^{\rm op}\mbox{-mod})^{\rm op}), \quad I\mapsto {\rm Hom}_{\mathbf{D}^b(A^{\rm op}\mbox{-}{\rm mod})}(-, I).$$
Using the duality functor $D={\rm Hom}_k(-, k)$ on modules, we identify $\mathbf{K}^b(A\mbox{-proj})$ with $\mathbf{K}^b(A^{\rm op}\mbox{-inj})^{\rm op}$, and $\mathbf{D}^b(A\mbox{-mod})$ with $\mathbf{D}^b(A^{\rm op}\mbox{-mod})^{\rm op}$. Then the required equivalence follows immediately.
\end{proof}

\section{Pseudo-adjunctions and triangle autoequivalences}

In this section, we apply Theorem \ref{thm:1} to obtain a $2$-categorical duality between two strict $2$-categories involving the bounded homotopy categories of projective modules and the bounded derived categories of module categories, respectively. Since the triangulated structures are not properly captured in the equivalences in Theorem \ref{thm:1}, we use the pseudo-adjunctions in \cite{Ball} to obtain the required assignment between triangle functors.

Throughout this section,  $A$ and $B$ will be two finite dimensional $k$-algebras.

\subsection{Pseudo-adjunctions and a $2$-categorical duality}

Let $\mathcal{T}$ and $\mathcal{T}'$ be triangulated categories, with translation functors $\Sigma$ and $\Sigma'$, respectively. Recall that a triangle functor $(F, \omega)\colon \mathcal{T}\rightarrow \mathcal{T}'$ consists of an additive functor $F\colon \mathcal{T}\rightarrow \mathcal{T}'$ and a natural isomorphism $\omega\colon F\Sigma\rightarrow \Sigma' F$, called the \emph{connecting isomorphism} of $F$, such that it respects exact triangles; more precisely, any exact triangle $X\stackrel{f}\rightarrow Y \stackrel{g}\rightarrow Z \stackrel{h}\rightarrow \Sigma(X)$ in $\mathcal{T}$ is sent to an exact triangle $F(X)\stackrel{F(f)}\rightarrow F(Y) \stackrel{F(g)}\rightarrow F(Z) \xrightarrow{\omega_X\circ F(h)} \Sigma'F(X)$ in $\mathcal{T}'$. We will later suppress $\omega$ and denote $(F, \omega)$ simply by $F$. We emphasize that natural transformations between triangle functors are required to respect the connecting isomorphisms.

Let $F=(F, \omega)\colon \mathbf{K}^b(A\mbox{-proj})\rightarrow \mathbf{K}^b(B\mbox{-proj})$ be a triangle functor. For each complex $X\in \mathbf{D}^b(B\mbox{-mod})$, the following cohomological functor
$${\rm Hom}_{\mathbf{D}^b(B\mbox{-}{\rm mod})}(F(-), X)\colon \mathbf{K}^b(A\mbox{-proj})^{\rm op}\longrightarrow k\mbox{-mod}$$
is locally-finite. By Theorem \ref{thm:1}, there is a unique complex $F^\vee(X)\in \mathbf{D}^b(A\mbox{-mod})$ with a natural isomorphism
$${\rm Hom}_{\mathbf{D}^b(B\mbox{-}{\rm mod})}(F(-), X)\stackrel{\sim}\longrightarrow {\rm Hom}_{\mathbf{D}^b(A\mbox{-}{\rm mod})}(-, F^\vee(X))|_{\mathbf{K}^b(A\mbox{-}{\rm proj})}.$$
Moreover, this defines a $k$-linear functor $F^\vee\colon \mathbf{D}^b(B\mbox{-mod})\rightarrow \mathbf{D}^b(A\mbox{-mod})$ such that there is a $k$-linear bifunctorial isomorphism
$$\Phi_{P, X}\colon {\rm Hom}_{\mathbf{D}^b(B\mbox{-}{\rm mod})}(F(P), X)\stackrel{\sim}\longrightarrow {\rm Hom}_{\mathbf{D}^b(A\mbox{-}{\rm mod})}(P, F^\vee(X))$$
for all $P\in \mathbf{K}^b(A\mbox{-proj})$ and $X\in \mathbf{D}^b(B\mbox{-mod})$. The connecting isomorphism $\omega$ yields a natural isomorphism $\omega^\vee\colon F^\vee \Sigma\rightarrow \Sigma F^\vee$ by the following commutative diagram,
\[\xymatrix{
(\Sigma^{-1}F(P), X) \ar[d]^-{(\Sigma^{-1}\omega_{\Sigma^{-1}(P)}, X)}\ar[rr]^-{\Sigma} && (F(P), \Sigma(X)) \ar[rr]^-{\Phi_{P, \Sigma(X)}} && (P, F^\vee \Sigma(X))\ar[d]^-{(P, {\omega^\vee_X})}\\
(F\Sigma^{-1}(P), X) \ar[rr]^-{\Phi_{\Sigma^{-1}(P), X}} && (\Sigma^{-1}(P), F^\vee(X)) \ar[rr]^-{\Sigma} && (P, \Sigma F^\vee(X)),
}\]
where we omit the notation Hom in the Hom spaces.

\begin{lem}
Keep the notation as above. Then $F^\vee=(F^\vee, \omega^\vee)\colon \mathbf{D}^b(B\mbox{-}{\rm mod})\rightarrow \mathbf{D}^b(A\mbox{-}{\rm mod})$ is a triangle functor.
\end{lem}

Following \cite[Section 4]{Ball}, we call $F^\vee$ the \emph{right pseudo-adjoint} of $F$.

\begin{proof}
This is due to \cite[Lemma 4.11]{Ball}, where  we replace the locally-free resolutions of complexes of sheaves in the proof by the projective resolutions of complexes of modules.
\end{proof}

Conversely, for a triangle functor $G\colon \mathbf{D}^b(B\mbox{-}{\rm mod})\rightarrow \mathbf{D}^b(A\mbox{-}{\rm mod})$ and a complex $P\in \mathbf{K}^b(A\mbox{-proj})$, the following cohomological functor
$${\rm Hom}_{\mathbf{D}^b(A\mbox{-}{\rm mod})}(P, G(-))\colon \mathbf{D}^b(B\mbox{-}{\rm mod})\longrightarrow k\mbox{-mod}$$
is locally-finite. By Theorem \ref{thm:1} and a similar argument as above, we obtain a $k$-linear functor ${^\vee G}\colon \mathbf{K}^b(A\mbox{-proj})\rightarrow \mathbf{K}^b(B\mbox{-proj})$ and a bifunctorial isomorphism
$$\Psi_{P, X}\colon {\rm Hom}_{\mathbf{D}^b(A\mbox{-}{\rm mod})}(P, G(X))\stackrel{\sim}\longrightarrow {\rm Hom}_{\mathbf{D}^b(A\mbox{-}{\rm mod)}}({^\vee G}(P), X)$$
for all $P\in \mathbf{K}^b(A\mbox{-proj})$ and $X\in \mathbf{D}^b(B\mbox{-mod})$. Moreover, by \cite[Lemma 4.13]{Ball}, the functor ${^\vee G}$ is a triangle functor, called the \emph{left pseudo-adjoint} of $G$. We call the above isomorphisms $\Phi$ and $\Psi$ \emph{pseudo-adjunctions}.

We denote by $\mathbb{K}^b$ the strict $2$-category,  whose objects are all the finite dimensional $k$-algebras $A$ such that $1$-morphisms are triangle functors between their bounded homotopy categories $\mathbf{K}^b(A\mbox{-proj})$ of projective modules and that $2$-morphisms are natural transformations between these triangle functors. Similarly, we have the strict $2$-category $\mathbb{D}^b$ by replacing $\mathbf{K}^b(A\mbox{-proj})$ with $\mathbf{D}^b(A\mbox{-mod})$. Denote by $(\mathbb{D}^b)^{\rm coop}$ the \emph{bidual} of $\mathbb{D}^b$, where both the $1$-morphisms and $2$-morphisms are reversed.

The analogue of the following result for projective schemes is essentially proved in \cite[Section 4]{Ball}.

\begin{thm} \label{thm:2} The assignment $F\mapsto F^\vee$ gives rise to a $2$-equivalence
$$\mathbb{K}^b\stackrel{\sim}\longrightarrow (\mathbb{D}^b)^{\rm coop},$$
which acts on objects by the identity and whose inverse is given by the assignment $G\mapsto {^\vee G}$.
\end{thm}

\begin{proof}
Using the pseudo-adjunctions, the assignment $F\mapsto F^\vee$ defines a (non-strict) $2$-functor $\mathbb{K}^b\longrightarrow (\mathbb{D}^b)^{\rm coop}$, whose action on objects is the identity. By  the following bifunctorial isomorphisms
$$(F(P), X)\stackrel{\Phi_{P, X}}\longrightarrow (P, F^\vee(X)) \stackrel{\Psi_{P, X}}\longrightarrow ({^\vee(F^\vee)}(P), X)$$
we obtain an isomorphism $F\rightarrow {^\vee(F^\vee)}$. Similarly, we obtain an isomorphism $G\rightarrow (^\vee G)^\vee$ for each $1$-morphism $G$ in $\mathbb{D}^b$. Then it is routine to verify that we have the required mutually inverse $2$-equivalences.
\end{proof}

\subsection{Extending functors and equivalences}

We will extract useful information from the $2$-equivalence in Theorem \ref{thm:2}. The treatment here is inspired by \cite[Lemmas 4.5 and 4.6]{Ball} with substantial difference.

\begin{lem}\label{lem:adj}
Let $F\colon \mathbf{K}^b(A\mbox{-}{\rm proj})\rightarrow \mathbf{K}^b(B\mbox{-}{\rm proj})$ and $F_1\colon \mathbf{K}^b(B\mbox{-}{\rm proj}) \rightarrow \mathbf{K}^b(A\mbox{-}{\rm proj})$ be two triangle functors. Then the following statements hold.
\begin{enumerate}
\item The pair $(F, F_1)$ is adjoint if and only if $(F^\vee, F_1^\vee)$ is adjoint.
\item The functor $F$ is an equivalence if and only if so is $F^\vee$.
\end{enumerate}
\end{lem}

\begin{proof}
We recall that a $2$-equivalence preserves adjoint $1$-morphisms and internal equivalences. Then the results follow from Theorem \ref{thm:2}.
\end{proof}

Let $F\colon \mathbf{K}^b(A\mbox{-}{\rm proj})\rightarrow \mathbf{K}^b(B\mbox{-}{\rm proj})$ be a triangle functor. We say that a triangle functor $\tilde{F}\colon \mathbf{D}^b(A\mbox{-}{\rm mod})\rightarrow \mathbf{D}^b(B\mbox{-}{\rm mod})$ \emph{extends} $F$,  provided that $\tilde{F}(\mathbf{K}^b(A\mbox{-}{\rm proj}))\subseteq \mathbf{K}^b(B\mbox{-}{\rm proj})$ and that $F$ is isomorphic to the restriction $\tilde{F}|_{\mathbf{K}^b(A\mbox{-}{\rm proj})}$ as triangle functors.

We mention that the following is proved in \cite[Lemma 2.8]{AKLY} under the additional assumption that $F$ is given by the tensor product of a certain bounded complex of bimodules.

\begin{prop}\label{prop:extension}
Let $F\colon \mathbf{K}^b(A\mbox{-}{\rm proj})\rightarrow \mathbf{K}^b(B\mbox{-}{\rm proj})$ be a triangle functor. Then $F$ admits an extension  $\tilde{F}\colon \mathbf{D}^b(A\mbox{-}{\rm mod})\rightarrow \mathbf{D}^b(B\mbox{-}{\rm mod})$ if and only if $F$ has a left adjoint.

In this situation, the extension $\tilde{F}$ of $F$ is unique up to isomorphism, which necessarily has a right adjoint. Moreover, $F$ is an equivalence if and only if so is its extension $\tilde{F}$.
\end{prop}

\begin{proof}
For the ``only if" part of the first statement, we assume that $\tilde{F}$ extends $F$. For each $Q\in \mathbf{K}^b(B\mbox{-}{\rm proj})$ and $P\in \mathbf{K}^b(A\mbox{-}{\rm proj})$, we have bifunctorial isomorphisms
$$(Q, F(P))\stackrel{\sim}\longrightarrow (Q, \tilde{F}(P)) \stackrel{\Psi_{Q, P}}\longrightarrow ({^\vee \tilde{F}}(Q), P).$$
This yields the required adjunction.

For the ``if" part, we assume that $(F_1, F)$ is an adjoint pair. Then by Lemma~\ref{lem:adj}(1), we have an adjoint pair $(F_1^\vee, F^\vee)$. For each $P\in \mathbf{K}^b(A\mbox{-}{\rm proj})$ and $X\in \mathbf{D}^b(B\mbox{-mod})$, we have bifunctorial isomorphisms
$$(F(P), X)\stackrel{\Phi_{P, X}}\longrightarrow (P, F^\vee(X)) \stackrel{\sim}\longrightarrow (F_1^\vee(P), X).$$
By Yoneda's Lemma, we have a natural isomorphism $F(P)\simeq F_1^\vee(P)$, that is, $F_1^\vee$ extends $F$.

For the uniqueness of $\tilde{F}$, we observe that ${^\vee(\tilde{F})}$ is isomorphic to the left adjoint $F_1$ of $F$. It follows that $\tilde{F}\simeq F_1^\vee$, in particular, it admits a right adjoint $F^\vee$. If $F$ is an equivalence, then $F_1$ and thus $F_1^\vee$ are equivalences. This proves the ``only if" part of the last statement. The ``if" part is well known; see Lemma \ref{lem:perf}.
\end{proof}

An  analogue of the following result for  projective schemes is mentioned in \cite[Remark 4.6]{Ball} with a different argument.

\begin{cor}
Let $G\colon \mathbf{D}^b(A\mbox{-}{\rm mod})\rightarrow \mathbf{D}^b(B\mbox{-}{\rm mod})$ be a triangle functor. Then $G$ has a right adjoint if and only if $G(\mathbf{K}^b(A\mbox{-}{\rm proj}))\subseteq \mathbf{K}^b(B\mbox{-}{\rm proj})$.
\end{cor}

\begin{proof}
The ``only if" part is well known. Assume that $G$ has a right adjoint $G_1$. Let $P\in \mathbf{K}^b(A\mbox{-}{\rm proj})$. The adjunction
$${\rm Hom}_{\mathbf{D}^b(B\mbox{-}{\rm mod})}(G(P), \Sigma^i(X))\simeq {\rm Hom}_{\mathbf{D}^b(A\mbox{-}{\rm mod})}(P, \Sigma^iG_1(X))$$
implies that ${\rm Hom}_{\mathbf{D}^b(B\mbox{-}{\rm mod})}(G(P), \Sigma^i(X))=0$ for each $X\in \mathbf{D}^b(B\mbox{-}{\rm mod})$ and $i\gg 0$. In view of Lemma \ref{lem:perf}, the complex $G(P)$ lies in $\mathbf{K}^b(B\mbox{-proj})$.

For the ``if" part, we denote by $F=G|_{\mathbf{K}^b(A\mbox{-}{\rm proj})}\colon \mathbf{K}^b(A\mbox{-}{\rm proj})\rightarrow \mathbf{K}^b(B\mbox{-}{\rm proj})$ the restriction of $G$. In particular, $G$ extends $F$. Then the existence of the right adjoint is proved in Proposition \ref{prop:extension}.
\end{proof}

For a triangulated category $\mathcal{T}$, ${\rm Aut}_\vartriangle(\mathcal{T})$ denotes its triangle autoequivalence group, which consists of the isomorphism classes of triangle autoequivalences on $\mathcal{T}$ and whose multiplication is induced by the composition of autoequivalences.

\begin{cor}\label{cor:auto}
The restriction  homomorphism between triangle autoequivalence groups
$${\rm Aut}_\vartriangle(\mathbf{D}^b(A\mbox{-}{\rm mod})) \stackrel{\rm res}\longrightarrow {\rm Aut}_\vartriangle(\mathbf{K}^b(A\mbox{-}{\rm proj})), \quad G\mapsto G|_{\mathbf{K}^b(A\mbox{-}{\rm proj})}$$
is an isomorphism.
\end{cor}

\begin{proof}
Since $G$ extends $G|_{\mathbf{K}^b(A\mbox{-}{\rm proj})}$, the required injectivity follows from the uniqueness of the extension functor in Proposition \ref{prop:extension}. On the other hand, we infer from Proposition \ref{prop:extension} that each triangle autoequivalence on $\mathbf{K}^b(A\mbox{-}{\rm proj})$ extends to a triangle autoequivalence on $\mathbf{D}^b(A\mbox{-}{\rm mod})$. This implies the required surjectivity.
\end{proof}

We mention that Corollary \ref{cor:auto} is related to the work \cite{CY}.

Recall from \cite{Ye,RZ} that ${\rm DPic}(A)$ is the \emph{derived Picard group} of $A$, whose elements are the isomorphism classes of two-sided tilting complexes of $A$-modules and whose multiplication is given by the derived tensor product over $A$. The evaluation homomorphism
$${\rm ev}\colon {\rm DPic}(A) \longrightarrow {\rm Aut}_\vartriangle(\mathbf{D}^b(A\mbox{-mod}))$$
sends a two-sided tilting complex $X$ to the derived tensor functor $X\otimes_A^\mathbb{L}-$.

Recall from \cite{CY} that an additive category $\mathcal{P}$ is \emph{$\mathbf{K}$-standard} if the following condition is satisfied: each triangle autoequivalence $F$ on $\mathbf{K}^b(\mathcal{P})$ is isomorphic to the identity functor as triangle functors, provided that it satisfies $F(\mathcal{P})\subseteq \mathcal{P}$ and that $F|_\mathcal{P}\colon \mathcal{P}\rightarrow \mathcal{P}$ is isomorphic to the identity functor. Similarly, an abelian category $\mathcal{A}$ is \emph{$\mathbf{D}$-standard} if the following condition is satisfied: each triangle autoequivalence $F$ on $\mathbf{D}^b(\mathcal{A})$ is isomorphic to the identity functor as triangle functors, provided that it satisfies $F(\mathcal{A})\subseteq \mathcal{A}$ and that $F|_\mathcal{A}\colon \mathcal{A}\rightarrow \mathcal{A}$ is isomorphic to the identity functor.

The following known results are one of the main motivations for these concepts.

\begin{lem}\label{lem:stand}
The following statements hold.
\begin{enumerate}
\item The module category $A\mbox{-}{\rm mod}$ is $\mathbf{D}$-standard if and only if the evaluation homomorphsim ``${\rm ev}$" is surjective.
\item The category $A\mbox{-}{\rm proj}$ is $\mathbf{K}$-standard if and only if the composition  ``${\rm res}\circ {\rm ev}$" is surjective.
\end{enumerate}
\end{lem}

\begin{proof}
Recall that the surjectivity of ``${\rm ev}$" is equivalent to the condition that every derived autoequivalences on $\mathbf{D}^b(A\mbox{-mod})$ is standard. Then (1) is contained in \cite[Theorem 5.10]{CY}. By a similar argument for $\mathbf{K}^b(A\mbox{-proj})$, one proves (2).
\end{proof}

Combing Corollary \ref{cor:auto} and Lemma \ref{lem:stand}, we have the following immediate consequence.

\begin{cor}\label{cor:stand}
Let $A$ be a finite dimensional $k$-algebra. Then $A\mbox{-}{\rm proj}$ is $\mathbf{K}$-standard if and only if $A\mbox{-}{\rm mod}$ is $\mathbf{D}$-standard. \hfill $\square$
\end{cor}

We mention that the ``only if" part is known. Indeed, let $\mathcal{A}$ be an abelian category with enough projective objects. Denote by $\mathcal{P}$ its full subcategory formed by projective objects. By \cite[Theorem 6.1]{CY}, the $\mathbf{K}$-standardness of $\mathcal{P}$ implies the $\mathbf{D}$-standardness of $\mathcal{A}$. In view of  Corollary \ref{cor:stand}, we expect that the inverse implication is true. This is related to the following question: does any triangle autoequivalence on $\mathbf{K}^b(\mathcal{P})$ extend to a triangle autoequivalence on $\mathbf{D}^b(\mathcal{A})$?

\vskip 10pt

\noindent {\bf Acknowledgements}.\quad The author thanks Jian Liu and Dong Yang for helpful comments. This work is supported by the National Natural Science Foundation of China (Nos. 11522113 and 11671245),  the Fundamental Research Funds for the Central Universities,  and Anhui Initiative in Quantum Information Technologies (AHY150200).

\bibliography{}

\begin{thebibliography}{9999}

\bibitem{AKLY} {\sc L. Angeleri H\"{u}gel, S. Koenig, Q. Liu, and D. Yang}, {\em Ladders and simplicity of derived module categories}, J. Algebra {\bf 472} (2017), 15--66.




\bibitem{Ball} {\sc M.R. Ballard}, {\em Derived categories of sheaves on singular schemes with an application to reconstruction}, Adv. Math. {\bf 227} (2011), 895--919.

\bibitem{BV} {\sc A.I. Bondal, and M. Van den Bergh}, {\em Generators and representability of functors in commutative and noncommutative geometry}, Mosc. Math. J. {\bf 3}(1) (2001), 327--344.


\bibitem{Buc} {\sc R.O. Buchweitz,}  Maximal Cohen-Macaulay Modules
and Tate Cohomology over Gorenstein Rings, Unpublished Manuscript,
1987.

\bibitem{CY} {\sc X.W. Chen, and Y. Ye}, {\em The $\mathbf{D}$-standard and $\mathbf{K}$-standard categories}, Adv. Math. {\bf 333} (2018), 159--193.

\bibitem{CKN} {\sc J.D. Christensen, B. Keller, and A. Neeman}, {\em Failure of Brown representability in derived categories}, Topology {\bf 40}(6) (2001), 1339--1361.

\bibitem{Kr00} {\sc H. Krause}, {\em Smashing subcategories and the telescope conjecture--an algebraic approach}, Invent. Math. {\bf 139}(1) (2000), 99--133.

\bibitem{Kr} {\sc H. Krause}, {\em The stable derived category of a noetherian scheme}, Compo. Math. {\bf 141} (2005), 1128--1162.

\bibitem{KrYe} {\sc H. Krause, and Y. Ye}, {\em On the centre of a triangulated category}, Proc. Edin. Math. Soc. {\bf 54} (2011), 443--466.

\bibitem{Or04} {\sc D. Orlov}, {\em Triangulated categories of singularities and D-branes in Landau-Ginzburg
models}, Trudy Steklov Math. Institute {\bf 204} (2004), 240--262.

\bibitem{Ric} {\sc J. Rickard}, {\em Derived equivalences as derived functors}, J. London Math. Soc. {\bf 43}(2) (1991), 37--48.

\bibitem{Rou} {\sc R. Rouquier}, {\em Dimensions of triangulated categories}, J. K-Theory {\bf 1}(2) (2008), 193--256.

\bibitem{RZ} {\sc R. Rouquier, and A. Zimmermann}, {\em  Picard groups for derived module categories}, Proc. Lond. Math. Soc. {\bf 87}(1) (2003), 197--225.

\bibitem{Ye} {\sc A. Yekutieli}, {\em Dualizing complexes, Morita equivalence and the derived Picard group of a ring} J. Lond. Math. Soc. {\bf 60}(3) (1999), 723--746.

\end{thebibliography}

\vskip 10pt

 {\footnotesize \noindent Xiao-Wu Chen\\
 Key Laboratory of Wu Wen-Tsun Mathematics, Chinese Academy of Sciences,\\
 School of Mathematical Sciences, University of Science and Technology of China, Hefei 230026, Anhui, PR China}

\end{document}